\documentclass[11pt,reqno]{amsart}
\usepackage{amsmath,amssymb,amsthm}
\usepackage{graphicx}
\usepackage[all]{xypic}
\usepackage[utf8,nocaptions]{vietnam}
\usepackage{calrsfs}
\input xypic
\def\thmsection{section}
\def\thmchangesection{changesection}

\def\thmchangechapter{changechapter}
\def\thmchange{change}
\def\thmplain{plain}
\ifx\theoremnumberstyle\thmplain
  \theoremstyle{break-italic}
  \newtheorem{satz}{Satz}
\else
  \ifx\theoremnumberstyle\thmsection
    \theoremstyle{break-italic}
    \newtheorem{satz}{Satz}[section]
  \else
    \ifx\theoremnumberstyle\thmchange
      \swapnumbers
      \theoremstyle{break-italic}
      \newtheorem{satz}{Satz}
    \else
      \ifx\theoremnumberstyle\thmchangesection
         \swapnumbers
         \theoremstyle{break-italic}
         \newtheorem{satz}{Satz}[section]
      \else
        \ifx\theoremnumberstyle\thmchangechapter
           \swapnumbers
           \theoremstyle{break-italic}
           \newtheorem{satz}{Satz}[chapter]
        \else
          \ifx\theoremnumberstyle\thmsection
             \theoremstyle{break-italic}
             \newtheorem{satz}{Satz}[section]
          \else
            \theoremstyle{break-italic}
            \newtheorem{satz}{Satz}[section]
          \fi
        \fi
      \fi
    \fi
  \fi
\fi

\theoremstyle{break-italic}
\newtheorem{theorem}[satz]{Theorem}
\newtheorem{lemma}[satz]{Lemma}

\newtheorem{corollary}[satz]{Corollary}
\newtheorem{proposition}[satz]{Proposition}

\newtheorem*{conjecture*}{Conjecture}

\theoremstyle{break-roman}
\newtheorem{definition}[satz]{Definition}
\newtheorem{problem}[satz]{Problem}

\theoremstyle{standard}

\theoremstyle{varthm-roman}
\newtheorem*{varthm-roman}{}
\theoremstyle{varthm-italic}
\newtheorem*{varthm-italic}{}
\theoremstyle{varthm-roman-break}
\newtheorem*{varthm-roman-break}{}
\theoremstyle{varthm-italic-break}
\newtheorem*{varthm-italic-break}{}
\theoremstyle{varthm-roman-no-punctuation}
\newtheorem{varthm-roman-no-punctuation-numbered}[satz]{}
\theoremstyle{varthm-italic-no-punctuation}
\newtheorem{varthm-italic-no-punctuation-numbered}[satz]{}

\newenvironment{varthm-roman-numbered}[1]{
  \begin{varthm-roman-no-punctuation-numbered}
    \mbox{\rm\textbf{#1}}
  }{\end{varthm-roman-no-punctuation-numbered}}
\newenvironment{varthm-italic-numbered}[1]{
  \begin{varthm-italic-no-punctuation-numbered}
    \mbox{\rm\textbf{#1}}
  }{\end{varthm-italic-no-punctuation-numbered}}
\newenvironment{varthm-roman-break-numbered}[1]{
  \begin{varthm-roman-no-punctuation-numbered}
    \mbox{\rm\textbf{#1}\newline}
  }{\end{varthm-roman-no-punctuation-numbered}}
\newenvironment{varthm-italic-break-numbered}[1]{
  \begin{varthm-italic-no-punctuation-numbered}
    \mbox{\rm\textbf{#1}}\newline
  }{\end{varthm-italic-no-punctuation-numbered}}
\numberwithin{equation}{section}

\def\eex{{\accent"5E e}\kern-.385em\raise.2ex\hbox{\char'23}\kern-.08em}
\def\EES{{\accent"5E E}\kern-.5em\raise.8ex\hbox{\char'23 }}
\def\ow{o\kern-.42em\raise.82ex\hbox{
\vrule width .12em height .0ex depth .075ex \kern-0.16em \char'56}\kern-.07em}
\def\OW{O\kern-.460em\raise1.36ex\hbox{
\vrule width .13em height .0ex depth .075ex \kern-0.16em \char'56}\kern-.07em}

\def\bF{\bold{F}}
\def\bA{\bold{A}}
\def\bG{\bold{G}}
\def\bH{\bold{H}}
\def\bE{\bold{E}}
\def\bS{\bold{S}}
\def\bI{\bold{I}}
\def\bW{\bold{W}}
\def\tr{\mbox{tr}}

\begin{document}

\title[Tracial moment problems on  hypercubes ]{Tracial moment problems on   hypercubes}

 \author{Cong Trinh Le }
\address{Department of Mathematics, Quy Nhon University, 170 An Duong Vuong, Quy Nhon, Binh Dinh, Vietnam}
\email{lecongtrinh@qnu.edu.vn}


\date{\today}



\begin{abstract}  In this paper we  introduce  the  \textit{tracial $K$-moment problem}  and the \textit{sequential matrix-valued $K$-moment problem} and show the equivalence of the solvability of  these problems.  Using a Haviland's theorem for matrix polynomials, we solve these  $K$-moment problems for the case where $K$ is the hypercube $[-1,1]^n$. 
\end{abstract}

\maketitle
\section{Introduction}
Let $\mathbb R [X]:=\mathbb R[X_1,\ldots,X_n]$ denote the  algebra  of polynomials in $n$ variables $X_1,\ldots, X_n$ with real coefficients. 

For a  real-valued linear functional $L$ on $\mathbb R[X]$ and  a closed subset $K\subseteq \mathbb R^n$, the \textit{$K$-moment problem} asks \textit{when does  there exist a positive Borel measure $\sigma$ on $\mathbb R^n$ supported on  $K$ such that} 
$$ L(f) = \int f(x) d\sigma(x), ~ \forall f\in \mathbb R[X]? $$
Each functional of this form is called a \textit{$K$-moment functional}.  

Another kind of $K$-moment problems for sequences of numbers is stated as follows.   For a multi-index sequence $s=(s_\alpha)_{\alpha \in \mathbb N_0^n}$, the \textit{$K$-moment problem for sequences} asks \textit{when does there exist   a positive Borel measure $\sigma$ on $\mathbb R^n$ supported on $K$ such that }
$$ s_\alpha = \int x^\alpha d\sigma(x), ~\forall \alpha \in \mathbb N_0^n, $$
where $x^\alpha = x_1^{\alpha_1}\ldots x_n^{\alpha_n}$. Each sequence of this form is called a \textit{$K$-moment sequence}.

For each sequence $s=(s_\alpha)_{\alpha \in \mathbb N_0^n}$, the \textit{Riesz   functional} $L_s$ on $\mathbb R[X]$ associated to $s$ is defined by 
$$L_s(x^\alpha)=s_{\alpha}, \forall \alpha \in \mathbb N_0^n. $$
It is easy to check that $s$ is a $K$-moment sequence if and only if $L_s$ is a $K$-moment functional. 

In the case $K=\mathbb R$ (resp. $[0, +\infty)$, $[0,1]$), we go back to the \textit{Hamburger} (resp. \textit{Stieltjes}, \textit{Hausdorff}) \textit{moment problem} which was solved completely by many authors (see, for example, a very excellent book of Schm\"udgen  \cite{Schm2} on the moment problem).  If $K$ is a \textit{compact} basic closed semi-algebraic set in $\mathbb R^n$, the $K$-moment problem was solved by Schm\"udgen (1991, \cite{Schm1}).

One of the power tools to solve $K$-moment problems is the \textit{Haviland theorem}, which was proved by E.K. Haviland (1935) and stated as follows. 

\textit{Given a linear functional $L$ on $\mathbb R[X]$ and a closed subset $K$ in $\mathbb R^n$. Then $L$ is a $K$-moment functional if and only if $L(f)\geq 0$ for all $f\geq 0$ on $K$.}

Here the notation $f\geq 0$ on $K$ means that $f(x)\geq 0, \forall x\in K$. A similar meaning is defined for the notation $f >0$ on $K$.

For a fix integer $t>0$, we denote by $M_t(\mathbb R[X])$ the algebra of $t\times t$ matrices with entries in $\mathbb R [X]$, and by ${S}_t(\mathbb R[X])$ the subalgebra of symmetric matrices. Each element $\bold{A} \in  {M}_t(\mathbb R[X])$ is a matrix whose entries are polynomials in $\mathbb R[X]$, which is called  a \emph{polynomial matrix}. $\bold{A}$ is also called a \emph{matrix polynomial},  because it can be viewed as a polynomial in $X_1,\ldots, X_n$ whose coefficients come from ${M}_t(\mathbb R)$. Namely, we can write $\bold{A}$ as 
$$  \bold{A}(X)=\sum_{|\alpha|=0}^{d} \bold{A}_{\alpha}X^{\alpha}, $$
where $\alpha=(\alpha_1,\cdots,\alpha_n) \in \mathbb N_0^{n}$, $|\alpha|:=\alpha_1+\ldots + \alpha_n$, $X^{\alpha}:=X_1^{\alpha_1}\ldots X_n^{\alpha_n}$, $\bold{A}_\alpha \in  {M}_t(\mathbb R)$.  To unify notation, throughout the paper each element of $ {M}_t(\mathbb R[X])$ is called a \emph{matrix polynomial}.    

Let $  \bold{A}(X)=\displaystyle\sum_{|\alpha|=0}^{d} \bold{A}_{\alpha}X^{\alpha} $ be a matrix polynomial. Denote 
$$ \bA_d(X):= \sum_{|\alpha|=d} \bA_\alpha X^\alpha.$$
 If $\bA_d \not \equiv 0$, the matrix polynomial $\bold{A}$ is called of \textit{degree  $d$}. In particular, when $n=1$ and $\bA(x)=\displaystyle\sum_{i=0}^d \bA_i x^i$ is a univariate  matrix polynomial of degree $d$,  the non-zero matrix $\bA_d$ is called the \textit{leading coefficient} of $\bA$. 

Recently, there are some researchs on (truncated) \textit{noncommutative  $K$-moment problems},  e.g.   Dette and Studden (2002, \cite{DS}), Burgdorf and Klep (2011, \cite{BK}), Burgdorf, Klep and Povh (2013, \cite{BKP}), Cimprič and Zalar (2013, \cite{CZ}), Kimsey and Woerdeman (2013, \cite{KiW}),  Kimsey (2011, \cite{Ki}) and the references therein. In particular, in \cite{CZ}, Cimprič and Zalar proved a version of the Haviland theorem for matrix polynomials, and they applied this development to solve the matrix-valued moment problem on compact basic closed semi-algbraic sets.

In the case where $K\subseteq \mathbb R^n$ is a basic closed semi-algebraic set defined by a   subset $G=\{g_1,\ldots,g_m\}$ of $\mathbb R[X]$, almost all solutions of the matrix-valued $K$-moment problems are obtained using  the positivity of the given linear functional on the quadratic module $M(G)^t$ in $M_t(\mathbb R[X])$, where $M(G)$ is the quadratic module  in $\mathbb R[X]$ generated by $G$ and 
$$ M(G)^t=\{\sum_{i=1}^r m_i \bA_i^T\bA_i | r\in \mathbb N, m_i \in M(G), \bA_i \in M_t(\mathbb R [X])\}. $$ 

The main aim of this paper is to introduce some kinds of $K$-moment problems for matrix polynomials and  solve these matrix-valued $K$-moment problems in the case where $K$ is  the hypercube $[-1,1]^n$. The most advantage to work on the hypercube $[-1,1]^n$ (and in general, on convex, compact polyhedra with non-empty interior) is that any  matrix polynomial which is positive definite on $[-1,1]^n$ can be represented in terms of positive definite scalar matrices (see Theorem \ref{thr-Bernstein1} for the case $n=1$ and Corollary  \ref{coro-Positive-on[-1,1]^n} for $n\geq 1$).  Then the solution of the matrix-valued $[-1,1]^n$-moment problem is obtained basing on the positivity of the given linear functional on the cone of positive definite scalar matrices, which is well described.

The paper is organized as follows. In Section 2 we introduce the definition of the \textit{tracial $K$-moment problem} and the  \textit{sequential matrix-valued $K$-moment problem} and establish an equivalence of the solvability of  these  problems (Proposition \ref{prop-relation}).  Moreover, we also prove in this section a version of Haviland's theorem for matrix polynomials which was  proved mainly by Cimprič and Zalar in \cite{CZ}. Next, in Section 3 we solve the tracial (resp. sequential matrix-valued) $[-1,1]$-moment problem, which is established based on a representation of univariate  matrix polynomials positive definite on $[-1,1]$. A solution of the tracial (resp. sequential matrix-valued) $[0,1]$-moment problem is also given in this section. Finally, in Section 4 we solve the tracial (resp. sequential matrix-valued) moment problem on the general hypercube $[-1,1]^n$.

\textbf{Notation.} For any matrix $\bold{A}\in  {M}_t(\mathbb R[X])$, the  notation $\bold {A} \geq 0$ means $\bold{A}$ is \emph{positive semidefinite},  i.e. for each ${x} \in \mathbb R^{n}$, $v^{T}\bold{A}(x)v \geq 0$   for all  $v\in \mathbb R^t$; $\bold {A} >0 $ means $\bold{A}$ is \emph{positive definite}, i.e. for each  ${x} \in \mathbb R^{n}$, $v^{T}\bold{A}(x)v > 0$   for all  $v\in \mathbb R^t \setminus\{0\}$. 

\section{Haviland's theorem for matrix polynomials}
\subsection{Matrix-valued measures and integrals}
In this section we recall some basic notions of the matrix-valued measures and integrals on a closed subset of $\mathbb R^n$,  which can be seen from the thesis of Kimsey \cite{Ki}. 

Throughout this subsection, let $X\subseteq \mathbb R^n$ be a non-empty closed set.

\begin{definition} \rm  Denote by $\mathcal{B}(X)$ the smallest $\sigma$-algebra generated from the open (or equivalently closed) subsets of  $X$.  A measure $\sigma$ defined on $\mathcal{B}(X)$ is called a \textit{Borel measure}.  A Borel measure $\sigma$ on $\mathcal{B}(X)$ is called \textit{finite} if $\sigma(X) < +\infty$. Denote by $m(X)$ the set of all finite measures on $X$. 

A measure $\sigma \in m(X)$ is called \textit{positive}  if $\sigma(E)\geq 0$ for all $E\in \mathcal{B}(X)$. The set of all finite positive Borel  measures on $X$ is denoted by $m^+(X)$. 

For each $\sigma \in m(X)$, the \textit{support} of $\sigma$ is defined by
$$ supp (\sigma):= \overline{\{E\in \mathcal{B}(X) : |\sigma|(E) > 0\}},$$
where $|\sigma|(E):=|\sigma(E)|$ for all $E\in \mathcal{B}(X)$. 
\end{definition} 

\begin{definition} \rm Let $t$ be a positive integer. Let $\sigma_{ij} \in m(X)$, $i,j=1,\ldots,t$. Define the matrix-valued function
$ \bE: \mathcal{B}(X) \longrightarrow M_t(\mathbb R) $ by 
$$\bE(A):=(\sigma_{ij}(A))_{i,j=1,\ldots,t}\in M_t(\mathbb R), \forall A\in \mathcal{B}(X).$$ 
The matrix-valued function $\bE$ defined by this way is called a \textit{matrix-valued measure} on $X$. The set of all matrix-valued measure on $X$ is denoted by $M(X)$.  The set 
$$supp(\bE):=\bigcup_{i,j=1}^t supp(\sigma_{ij}) $$
is called the \textit{support of the matrix-valued measure $\bE$}. 

If  ${\sigma}_{ij} = \sigma_{ji}$ for all $i,j=1,\ldots,t$, we say that $\bE$ is a \textit{symmetric measure}. In addition, if for all $A\in \mathcal{B}(X)$ and for all  $v\in \mathbb R^t$ we have 
$ v^T \bE(A)v \geq 0$, then $\bE$ is called a \textit{positive semidefinite} matrix-valued measure. The set of all positive semidefinite matrix-valued measures on $X$ is denoted by $M^+(X)$.
\end{definition}

\begin{definition} \rm  Let $\bE=(\sigma_{ij}) \in M(X)$. A function  $f: X \rightarrow \mathbb R$ is called \textit{$\bE$-measurable } if $f$ is $\sigma_{ij}$-measurable  for every $i,j=1,\ldots,t$. 

Let $\bE=(\sigma_{ij}) \in M(X)$ and  $f: X \rightarrow \mathbb R$ be $\bE$-measurable. The \textit{matrix-valued integral} of $f$ with respect to the matrix measure $\bE$ is defined by 
$$ \int_X f(x)d\bE(x):= \Big(\int_Xf(x)d \sigma_{ij}(x)\Big)_{i,j=1,\ldots,t} \in M_t(\mathbb R). $$
\end{definition}

\subsection{Tracial $K$-moment problems}
Let $\mathcal{L}$ be a real-valued linear functional on $M_t(\mathbb R[X])$, where $\mathbb R[X]:= \mathbb R[X_1,\ldots,X_n]$ and $K\subseteq \mathbb R^n$ a closed subset. 

\begin{definition}[{cf. \cite{CZ}}] \rm  $\mathcal{L}$ is called a \textit{tracial $K$-moment functional } if there exists a positive semidefinite matrix-valued $\bE\in M^+(\mathbb R^n)$  such that 
\begin{equation}\label{equ-tracial-functional}
 supp(\bE) \subseteq K \mbox{ and } \mathcal{L}(\bF) = \int \tr(\bF(x)d \bE(x)), ~ \forall \bF \in M_t(\mathbb R[X]).
\end{equation}
The matrix-valued measure $\bE\in M^+(\mathbb R^n)$ satisfying (\ref{equ-tracial-functional}) is called a \textit{representing measure} of the tracial $K$-moment functional $\mathcal{L}$.
\end{definition}

\begin{problem}\label{problem-tracial-functional} Let   $\mathcal{L}$ be a linear functional on $M_t(\mathbb R[X])$ and $K\subseteq \mathbb R^n$ a closed set. \textit{\textbf{The tracial $K$-moment problem }} asks 
 when does there exist a matrix-valued measure $\bE\in M^+(\mathbb R^n)$ satisfying the conditions (\ref{equ-tracial-functional})? 
\end{problem}

The following definition of matrix-valued $K$-moment sequences is learned from Kimsey  in \cite{Ki}.

\begin{definition} \rm Let $\bS=(\bS_\alpha)_{\alpha \in \mathbb N_0^n}$ be a multi-indexed sequence of symmetric matrices in $M_t(\mathbb R)$ and $K\subseteq \mathbb R^n$ a closed subset. The sequence $\bS$ is called a \textit{matrix-valued $K$-moment sequence } if 
there exists a positive semidefinite matrix-valued $\bE\in M^+(\mathbb R^n)$  such that 
\begin{equation}\label{equ-tracial-sequence}
 supp(\bE) \subseteq K \mbox{ and } \bS_\alpha = \int x^\alpha      d \bE(x), ~ \forall \alpha \in \mathbb N_0^n,
\end{equation}
where $x^\alpha=x_1^{\alpha_1}\ldots x_n^{\alpha_n}$ with $\alpha=(\alpha_1,\ldots,\alpha_n)$ and $x=(x_1,\ldots,x_n)\in \mathbb R^n$.  
 
 The matrix-valued measure $\bE\in M^+(\mathbb R^n)$ satisfying (\ref{equ-tracial-sequence}) is called a \textit{representing measure} of the   matrix-valued $K$-moment sequence $\bS$.
\end{definition}

\begin{problem}\label{problem-tracial-functional}  Let $\bS=(\bS_\alpha)_{\alpha \in \mathbb N_0^n}$ be a multi-indexed sequence of symmetric matrices in $M_t(\mathbb R)$ and $K\subseteq \mathbb R^n$ a closed set. 
\textit{\textbf{The sequential matrix-valued $K$-moment problem}} asks when does there exist a matrix-valued measure $\bE\in M^+(\mathbb R^n)$ satisfying the conditions (\ref{equ-tracial-sequence})?
\end{problem}

For each matrix-valued sequence $\bS=(\bS_\alpha)_{\alpha \in \mathbb N_0^n}$, let us consider the real-valued  function  $\mathcal{L}_\bS$ on $M_t(\mathbb R[X])$ defined as follows. For each $\bF(X)=\displaystyle\sum_{|\alpha|=0}^d \bF_\alpha X^\alpha \in M_t(\mathbb R[X])$, define
$$ \mathcal{L}_\bS(\bF):=\displaystyle\sum_{|\alpha|=0}^d\tr(\bF_\alpha \bS_\alpha). $$
It is easy to check that $\mathcal{L}_{\bS}$ is   \textit{linear}, which is called the \textit{Riesz functional associated to the sequence $\bS$}.  

\begin{proposition}\label{prop-relation} Let $\bS=(\bS_\alpha)_{\alpha \in \mathbb N_0^n}$ be a   sequence of symmetric matrices in $M_t(\mathbb R)$ and $K\subseteq \mathbb R^n$ a closed set.  Then $\bS$ is a matrix-valued $K$-moment sequence if and only if $\mathcal{L}_\bS$ is a tracial $K$-moment functional.
\end{proposition}
\begin{proof} Assume that $\bS$ is a matrix-valued $K$-moment sequence. Then there exists  a matrix-valued measure $\bE$ in $M^+(\mathbb R^n)$ supported in $K$ and 
$$\bS_\alpha = \int x^\alpha d\bE(x), ~\forall \alpha \in \mathbb N_0^n. $$
Then, for each $\bF(X)=\displaystyle\sum_{|\alpha|=0}^d \bF_\alpha X^\alpha\in M_t(\mathbb R[X])$, by definition of $\mathcal{L}_\bS$, we have
\begin{align*}
\mathcal{L}_\bS(\bF) &= \sum_{\alpha} \tr(\bF_\alpha \bS_\alpha)=\sum_\alpha \tr\Big(\int x^\alpha \bF_\alpha d\bE(x)\Big)\\
&=\int \tr\Big(\sum_\alpha \bF_\alpha x^\alpha d\bE(x)\Big) = \int \tr(\bF(x) d\bE(x)).
\end{align*}
This implies that $\mathcal{L}_\bS$ is a tracial $K$-moment functional.

Conversely, assume that $\bE$ is a matrix-valued measure in $M^+(\mathbb R^n)$ satisfying 
$$ supp(\bE) \subseteq K \mbox{ and } \mathcal{L}_\bS(\bF) = \int \tr(\bF(x)d \bE(x)), ~ \forall \bF \in M_t(\mathbb R[X]).$$
In particular, for each coordinate matrix $\bW_{k,l}$ of the algebra $M_t(\mathbb R)$ and each  $\alpha \in \mathbb N_0^n$, we have 
$$ \mathcal{L}_\bS(X^\alpha \bW_{k,l}) = \int x^\alpha \tr(\bW_{k,l}d \bE(x))=\int x^\alpha d\bE_{l,k}.  $$
Here we use the fact that for each $k,l=1,\ldots,t$ and for each matrix $\bA = (\bA_{i,j})_{i,j}\in M_t(\mathbb R)$, we have 
$$ \tr(\bW_{k,l}\bA) = \bA_{l,k}. $$
On the other hand, by definition of the associated Riesz functional, we have 
$$ \mathcal{L}_\bS(X^\alpha \bW_{k,l}) = \tr(\bW_{k,l}\bS_\alpha) = (\bS_{\alpha})_{l,k}.$$
It follows that 
$$
\bS_\alpha^T = \big((\bS_\alpha)_{l,k}\big)_{k,l} = \Big(\int x^\alpha d\bE_{l,k}(x)\Big)_{k,l} = \Big(\int x^\alpha d\bE(x)\Big)^T.
$$
Equivalently, $\bS_\alpha = \displaystyle\int x^\alpha d\bE(x)$, i.e. $\bS$ is a matrix-valued $K$-moment sequence.
\end{proof}

\subsection{Haviland's theorem for matrix polynomials}
\begin{theorem} \label{thr-Haviland}
Let $\mathcal{L}$ be a real-valued linear functional on $M_t(\mathbb R[X])$ and $K\subseteq \mathbb R^n$ a closed set. Then the following are equivalent.
\begin{itemize}
\item[(1)] $\mathcal{L}$ is a tracial $K$-moment functional, i.e. there exists a positive semi-definite matrix-valued measure $\bE$ on $\mathbb R^n$ whose support is contained in $K$ such that 
$$\mathcal{L}(\bF) = \int \tr(\bF d\bE). $$
\item[(2)] $\mathcal{L}(\bF) \geq 0$ for all $\bF\geq 0$ on $K$.
\item[(3)] $\mathcal{L}(\bF+\epsilon \bI) \geq 0$ for all $\bF\geq 0$ on $K$ and for all $\epsilon >0$.
\item[(4)] $\mathcal{L}(\bF) \geq 0$ for all $\bF > 0$ on $K$.
\end{itemize}
\end{theorem}
\begin{proof} The equivalence of (1) and (2) was proved by Cimprič and Zalar \cite[Theorem 3 and Remark 5]{CZ}.\\
$(2)\rightarrow (3)$ is obvious, since $\mathcal{L}(\bI)\geq 0$ by hypothesis.  Now we check
$(3)\rightarrow (2)$. Take $\bF \geq 0$ on $K$. Consider  $\epsilon = \dfrac{1}{n}$ for $n\in \mathbb N$. Then by the assumption of (3),
$$ 0\leq \mathcal{L}(\bF + \dfrac{1}{n}\bI)=\mathcal{L}(\bF) + \dfrac{1}{n}\mathcal{L}(\bI), $$
where $\mathcal{L}(\bI) \geq 0$ by hypothesis. Letting $n\rightarrow +\infty$, we get $\mathcal{L}(\bF) \geq 0$, i.e. we have (2).\\
It is obvious that $(2)\rightarrow (4)$. Now we verify $(4)\rightarrow (3)$. Take $\bF\geq 0$ on $K$ and $\epsilon >0$. Then $\bF + \epsilon \bI >0$ on $K$. Hence by the assumption we have $\mathcal{L}(\bF + \epsilon \bI) \geq 0$, i.e. we have (3). The proof is complete.
\end{proof}

\section{Tracial  moment problems on the intervals $[-1,1]$ and $[0,1]$}

Firstly we propose a version of the Bernstein theorem (cf. \cite[Prop. 3.4]{Schm2}) for matrix polynomials, representing a matrix polynomial positive definite on the interval $[-1,1]$. This is a special case of the Handelman's Positivstellensatz for matrix polynomials established in \cite{LeB}, but the proof given here is   special for the case of  one-variable matrix polynomials.

\begin{theorem} \label{thr-Bernstein1}Let $\bF \in M_t(\mathbb R[x])$ be a matrix polynomial of degree $d>0$. Assume $\bF (x) >0$ for all $x\in [-1,1]$. Then there exists a positive number $N\in \mathbb N$ and positive definite matrices $\bG_i, i=1,\ldots,N+d$ such that 
$$\bF(x) = \sum_{i=0}^{N+d} \bG_i (1+x)^i (1-x)^{N+d-i}.$$
\end{theorem}

The main idea of the proof is  learned from the proof \cite[Prop. 3.4]{Schm2} for polynomials, using  the \textit{Goursat transform} and the P\'olya theorem for homogeneous matrix polynomials established  in \cite{SchH}.

Let $\bF$ be as in Theorem \ref{thr-Bernstein1}. The \textit{Goursat transform} of $\bF$ is defined by 
$$ \tilde{\bF}(x):= (1+x)^d \bF(\dfrac{1-x}{1+x}).$$ 
\begin{lemma}\label{lm-Bernstein1}
$\deg(\tilde{\bF}) = d$, $\tilde{\bF}_d >0$, and $\tilde{\bF}(t) >0$ for all $t\in [0,+\infty)$, where $\tilde{\bF}_d$ denotes the leading  coefficient of the matrix polynomial $\tilde{\bF}$.
\end{lemma}
\begin{proof} Let $\bF(x)=\sum_{i=0}^d \bF_i x^i$, with $\bF_d\not \equiv 0$. Then we can write $\tilde{\bF}$ as 
$$ \tilde{\bF}(x)=\sum_{i=0}^d \bF_i (1-x)^i (1+x)^{d-i}.$$
The leading coefficient of   $\tilde{\bF}$ is
$$ \tilde{\bF}_d:=\sum_{i=0}^d (-1)^i \bF_i = \bF(-1) $$
which is positive definite by the assumption of $\bF$. 

For every $t\in [0,+\infty)$, $x:=\dfrac{1-t}{1+t}\in (-1,1]$. Observe that $t=\dfrac{1-x}{1+x}$, and hence
$$ \tilde{\bF}(t)=(1+t)^d \bF(x).$$
It follows from the assumption of $\bF$ that $\tilde{\bF}(t) >0$.
\end{proof}

\begin{lemma} \label{lm-Bernstein2} Let $\bG \in M_t(\mathbb R[x])$ be a matrix polynomial of degree $d$ with a positive definite leading coefficient $\bG_d$. If $\bG(x) >0$ for all $x\in [0,+\infty)$ then there exists a positive integer $N$ and positive definite matrices $\bH_i$, $i=0,\ldots,N+d$ such that 
$$ (1+x)^N\bG(x) = \sum_{i=0}^{N+d} \bH_i x^i.$$
\end{lemma}
\begin{proof} Assume $\bG(x)=\sum_{i=0}^d \bG_i x^i$. Denote by $\Delta$ the standard simplex in $\mathbb R^2$
$$\Delta = \{(x,y)\in \mathbb R^2 | x\geq 0, y\geq 0, x+y+1\}. $$
Let $ {\bG}^h(x,y) \in M_t(\mathbb R[x,y])$ be the \textit{ homogenization} of $\bG$ with respect to the new variable $y$, defined by
$$ \bG^h (x,y):=\sum_{i=0}^d \bG_i x^i y^{d-i}.$$
For any $(x,y)\in \Delta$,\\
if $y=0$, we have $x=1$, then
$$\bG^h(1,0)=\bG_d >0. $$
If $y>0$, we have 
$$\bG^h(x,y)=y^d\bG(\dfrac{x}{y}) >0, $$
since $\dfrac{x}{y} \in [0,+\infty)$. In summary, we have $\bG^h(x,y) >0$ for all $(x,y)\in \Delta$.

It follows from the P\'olya theorem for homogeneous matrix polynomials proved by Scherer and Hol \cite[Theorem 3]{SchH} that there exist a positive integer $N$ and positive definite matrices $\bG_{k,l},$ with $k,l\in \mathbb N_0$, $k+l=0,\ldots,N+d$ such that 
\begin{equation}\label{equ-polya1}
(x+y)^N \bG^h(x,y)= \sum_{k+l=N+d} \bG_{k,l} x^ky^l.
\end{equation} 
Substituting $y=1$ in both sides of the expression (\ref{equ-polya1}), observing that $\bG^h(x,1)=\bG(x)$, we get
$$ (1+x)^N \bG(x)= \sum_{k+l=N+d} \bG_{k,l} x^k = \sum_{k=0}^{N+d}G_{k,N+d-k}x^k.$$
For each $k=0,\ldots,N+d$, the matrices $\bH_k:=\bG_{k,N+d-k}$ satisfy conclusion of the lemma.
\end{proof}
Now we are ready to prove Theorem \ref{thr-Bernstein1}.
\begin{proof}[\textbf{Proof of Theorem \ref{thr-Bernstein1}}] Let $\bF$ satisfy the hypothesis of the theorem. It follows from Lemma \ref{lm-Bernstein1} that the Goursat transform $\tilde{\bF}$ satisfy the following properties
\begin{itemize}
\item $\deg(\tilde{\bF})=d$;
\item $\tilde{\bF}_d >0$;
\item $\tilde{\bF}(t) >0$ for all $t\in [0,+\infty)$.
\end{itemize}
Ir follows from Lemma \ref{lm-Bernstein2} that there exist a positive integer $N$ and positive definite matrices $\bH_{i}, i=0,\ldots,N+d$ such that 
\begin{equation}\label{equ-polya2}
(1+t)^{N}\tilde{\bF}(t) = \sum_{i=0}^{N+d} \bH_i t^i, ~ \forall t\in [0,+\infty). 
\end{equation}
For any $x\in (-1,1]$, there exists a unique $t\in [0,+\infty)$ such that $x=\dfrac{1-t}{1+t}$. Then 
$$t=\dfrac{1-x}{1+x}, ~(1+t)^{-1}=\dfrac{1+x}{2}, \mbox{ and } \tilde{\bF}(t)=(1+t)^d \bF(x). $$
Substituting into (\ref{equ-polya2}), we get
$$
(1+t)^N (1+t)^d \bF(x)  =  \sum_{i=0}^{N+d} \bH_i \big(\dfrac{1-x}{1+x}\big)^i.
$$
It follows that 
\begin{align*}
\bF(x) & = (1+t)^{-N+d} \sum_{i=0}^{N+d} \bH_i \big(\dfrac{1-x}{1+x}\big)^i \\
&=\dfrac{1}{2^{N+d}} (1+x)^{N+d} \sum_{i=0}^{N+d} \bH_i \big(\dfrac{1-x}{1+x}\big)^i\\
&=\sum_{i=0}^{N+d} \Big(\dfrac{1}{2^{N+d}}\bH_i\Big)(1-x)^i(1+x)^{N+d-i}.
\end{align*}
Hence the matrices $\bG_i:=\dfrac{1}{2^{N+d}}\bH_i$, $i=0,\ldots,N+d$, satisfy the conclusion of the theorem.
\end{proof}

It follows from the Haviland theorem for matrix polynomials (Theorem \ref{thr-Haviland}) and Theorem \ref{thr-Bernstein1} the following solution of the tracial $[-1,1]$-moment problem.

\begin{corollary}\label{coro-tracial-moment-on[-1,1]} Let $\mathcal{L}$ be a real-valued linear functional on $M_t(\mathbb R[x])$. Then the following are equivalent.
\begin{itemize}
\item[(1)] $\mathcal{L}$ is a tracial $[-1,1]$-moment functional.
\item[(2)] $\mathcal{L}((1+x)^k(1-x)^l\bG)\geq 0$ for all positive definite matrices $\bG \in M_t(\mathbb R)$ and $k,l\in \mathbb N_0$.
\end{itemize}
\end{corollary}

\begin{proof} If (1) holds, for each $\bG$ positive definite and $k,l \in \mathbb N_0$, since $(1+x)^k(1-x)^l\bG \geq 0  $ on   $[-1,1]$, it follows from Theorem   \ref{thr-Haviland}, $(1)\Rightarrow (2)$, that $\mathcal{L}((1+x)^k(1-x)^l\bG)\geq 0$. Hence we have $(1)\rightarrow (2)$. 

Conversely, assume (2) holds. Then for each $\bF \in S_t(\mathbb R[x])$ of degree $d$, $\bF> 0$ on $[-1,1]$, it follows from Theorem \ref{thr-Bernstein1} that 
$$ \bF(x)= \sum_{i=0}^{N+d} \bG_i (1+x)^i (1-x)^{N+d-i}$$
for some $N>0$ and $\bG_i >0$ for all $i=0,\ldots, N+d$.  Then 
$$\mathcal{L}(\bF) = \sum_{i=0}^{N+d} \mathcal{L}(\bG_i (1+x)^i (1-x)^{N+d-i}) \geq 0. $$
It follows from Theorem \ref{thr-Haviland}, $(4)\Rightarrow (1)$, that $\mathcal{L}$ is a tracial $[-1,1]$-moment functional, i.e. $(2) \rightarrow (1)$. The proof is complete.
\end{proof}

As a consequence of this result, we obtain the following  solution of the tracial $[0,1]$-moment problem, which was also solved by Cimprič and Zalar \cite[Coro. 1]{CZ}, using sums of Hermitian squares.
\begin{corollary}\label{coro-tracial-moment-on[0,1]} Let $\mathcal{L}$ be a real-valued linear functional on $M_t(\mathbb R[x])$. Then the following are equivalent.
\begin{itemize}
\item[(1)] $\mathcal{L}$ is a tracial $[0,1]$-moment functional.
\item[(2)] $\mathcal{L}(x^k(1-x)^l\bG)\geq 0$ for all positive definite matrices $\bG \in M_t(\mathbb R)$ and $k,l\in \mathbb N_0$.
\end{itemize}
\end{corollary}
\begin{proof}
Observe that there is a bijection from the interval $[-1,1]$ onto the interval $[0,1]$ given by 
$$x \mapsto \dfrac{x+1}{2}, \forall x \in [-1,1]. $$
Using this bijection, it follows from Theorem \ref{thr-Bernstein1} that for $\bF \in S_t(\mathbb R[x])$ of degree $d$, if $\bF > 0$ on $[0,1]$ then 
$$ \bF(x)=\sum_{i=0}^{N+d} \bG_i x^i (1-x)^{N+d-i}$$
for some $N>0$ and  $\bG_i >0$ for all $i=0,\ldots,N+d$. Then the result follows from this representation and Theorem \ref{thr-Haviland}.
\end{proof}

For the sequential matrix-valued $[-1,1]$-problem, we have the following solution.
\begin{corollary}\label{coro-tracial-sequence-on[-1,1]}
Let $\bS=(\bS_i)_{i\in \mathbb N_0}$ be a   sequence of symmetric matrices in $M_t(\mathbb R)$. Then $\bS$ is a matrix-valued  $[-1,1]$-moment sequence if and only if
$$ \sum_{i=0}^k\sum_{j=0}^l (-1)^j\binom{k}{i}\binom{l}{j} \tr(\bG\bS_{i+j}) \geq 0 $$
for all positive definite matrices $\bG\in M_t(\mathbb R)$ and $k,l\in \mathbb N_0$. 
\end{corollary}
\begin{proof}
Let $\mathcal{L}_\bS$ be the Riesz functional on $M_t(\mathbb R[x])$ associated to the sequence $\bS$.  For each $\bG\in S_t(\mathbb R[x])$, $\bG >0$, and $k,l \in \mathbb N_0$, we have 
\begin{align*}
\mathcal{L}_\bS \big((1+x)^k(1-x)^l\bG\big)&= \mathcal{L}_\bS\big(\sum_{i=0}^k \sum_{j=0}^l (-1)^j \binom{k}{i}\binom{l}{j}x^{i+j}\bG\big)\\
&=\sum_{i=0}^k \sum_{j=0}^l (-1)^j \binom{k}{i}\binom{l}{j}\mathcal{L}_\bS(x^{i+j}\bG)\\
&= \sum_{i=0}^k \sum_{j=0}^l (-1)^j \binom{k}{i}\binom{l}{j}\tr(\bG\bS_{i+j}).
\end{align*}
By Proposition \ref{prop-relation}, $\bS$ is a matrix-valued $[-1,1]$-moment sequence if and only if $\mathcal{L}_\bS$ is a tracial $[-1,1]$-moment functional. It follows from Corollary \ref{coro-tracial-moment-on[-1,1]} that this is equivalent to $\mathcal{L}_\bS \big((1+x)^k(1-x)^l\bG\big) \geq 0$ for all positive definite matrices $\bG\in M_t(\mathbb R)$ and $k,l\in \mathbb N_0$. This implies the result.
\end{proof}
Similarly, using Corollary \ref{coro-tracial-moment-on[0,1]}, we have the following solution of the sequential matrix-valued  $[0,1]$-moment problem.

\begin{corollary}\label{coro-tracial-sequence-on[0,1]}
Let $\bS=(\bS_i)_{i\in \mathbb N_0}$ be a   sequence of symmetric matrices in $M_t(\mathbb R)$. Then $\bS$ is a matrix-valued $[0,1]$-moment sequence if and only if 
$$ \sum_{i=0}^k (-1)^i\binom{k}{i}  \tr(\bG\bS_{i+l}) \geq 0 $$
for all positive definite matrices $\bG\in M_t(\mathbb R)$ and $k,l\in \mathbb N_0$.
\end{corollary}

\section{Tracial moment problems on the hypercube $[-1,1]^n$ }
In this section we consider a convex, compact polyhedron  $K\subseteq \mathbb R^n$   with non-empty interior. Assume that $K$ is the basic closed semi-algebraic set  defined by linear polynomials $L_1,\ldots, L_m \in \mathbb R[X]:=\mathbb R[X_1,\ldots,X_n]$, i.e.
$$ K=\{x\in \mathbb R^n | L_1(x)\geq 0, \ldots, L_m(x) \geq 0\}.$$
   The following version of Handelman's Positivstellensatz for matrix polynomials was proved by the the authors in \cite{LeB}. 
\begin{theorem}\label{thr-LeDu}
Let $\bF\in S_t(\mathbb R[X])$ be a matrix polynomial of degree $d>0$. If $\bF (x) > 0$ for all $x\in K$, then there exist a positive integer $N$ and positive definite matrices $\bG_\alpha \in M_t(\mathbb R)$, $\alpha \in \mathbb N_0^m$, $|\alpha|=0,\ldots, N+d$, such that 
$$ \bF(X)= \sum_{|\alpha|=0}^{N+d} \bG_\alpha L_1^{\alpha_1}\ldots L_m^{\alpha_m}. $$
\end{theorem}
Applying Theorem \ref{thr-LeDu} for the linear polynomials 
$$l_1= 1+X_1, l_2=1-X_1, \ldots, l_{2n-1}=1+X_n, l_{2n}=1-X_n $$
we obtain the following representation for matrix polynomials positive definite on the hypercube $[-1,1]^n$.

\begin{corollary} \label{coro-Positive-on[-1,1]^n}
Let $\bF\in S_t(\mathbb R[X])$ be a matrix polynomial of degree $d>0$. If $\bF (x) > 0$ for all $x\in [-1,1]^n$, then there exist a positive integer $N$ and positive definite matrices $\bG_\alpha \in M_t(\mathbb R)$, $\alpha \in \mathbb N_0^{2n}$, $|\alpha|=0,\ldots, N+d$, such that 
$$ \bF(X)= \sum_{|\alpha|=0}^{N+d} \bG_\alpha l_1^{\alpha_1}l_2^{\alpha_2}\ldots l_{2n-1}^{\alpha_{2n-1}} l_{2n}^{\alpha_{2n}}. $$
\end{corollary}

Now, applying the Haviland theorem for matrix polynomials (Theorem \ref{thr-Haviland}) and Corollary \ref{coro-Positive-on[-1,1]^n} we obtain the following solution of the tracial $[-1,1]^n$-moment problem.

\begin{corollary}\label{coro-tracial-moment-on[-1,1]^n}
Let $\mathcal{L}$ be a real-valued linear functional on $M_t(\mathbb R[X])$. Then the following are equivalent.
\begin{itemize}
\item[(1)] $\mathcal{L}$ is a tracial $[-1,1]^n$-moment functional.
\item[(2)] $\mathcal{L}(l_1^{\alpha_1}l_2^{\alpha_2}\ldots l_{2n-1}^{\alpha_{2n-1}} l_{2n}^{\alpha_{2n}}\bG)\geq 0$ for all positive definite matrices $\bG \in M_t(\mathbb R)$ and $\alpha=(\alpha_1,\ldots,\alpha_{2n})\in \mathbb N_0^{2n}$.
\end{itemize}
\end{corollary}

Observe that the solution of the tracial $K$-moment problem, where  $K$ is a convex, compact polyhedron  in $ R^n$   with non-empty interior, can be solved by the same way, using the linear polynomial $L_i$ instead of $l_j$ in Corollary \ref{coro-tracial-moment-on[-1,1]^n}.

Finally, applying Proposition \ref{prop-relation} and Corollary \ref{coro-tracial-moment-on[-1,1]^n} we obtain the following solution of the sequential matrix-valued $[-1,1]^n$-moment problem.

\begin{corollary}\label{coro-tracial-sequence-on[-1,1]^n}
Let $\bS=(\bS_{\alpha})_{\alpha\in \mathbb N_0^n}$ be a   sequence of symmetric matrices in $M_t(\mathbb R)$. Then $\bS$ is a matrix-valued  $[-1,1]^n$-moment sequence if and only if 
$$ \sum_{\beta\in \mathbb N_0^{2n}, \beta \leq \alpha} (-1)^{\mbox{e}(\alpha)}\binom{\alpha}{\beta}\tr(\bG\bS_{\beta}) \geq 0 $$
for all positive definite matrices $\bG\in M_t(\mathbb R)$ and $\alpha\in \mathbb N_0^{2n}$. 
\end{corollary}
Here, for $\alpha=(\alpha_1,\ldots,\alpha_{2n})$ and  $\beta=(\beta_1,\ldots,\beta_{2n})$,  
$\beta \leq \alpha$ means $\beta_i \leq \alpha_i$ for all $i=1,\ldots,2n$;  
$\binom{\alpha}{\beta}:=\binom{\alpha_1}{\beta_1}\ldots \binom{\alpha_{2n}}{\beta_{2n}}$; and  
$e(\alpha):=\displaystyle\sum_{i=1}^n\alpha_{2i}$. 
\begin{proof}
The proof is similar to that of Corollary \ref{coro-tracial-sequence-on[-1,1]}, using Proposition \ref{prop-relation}, Corollary \ref{coro-tracial-moment-on[-1,1]^n} and the identity 
$$ (1+X_1)^{\alpha_1}(1-X_1)^{\alpha_2}\ldots (1+X_n)^{\alpha_{2n-1}} (1-X_n)^{\alpha_{2n}} =  \sum_{\beta\in \mathbb N_0^{2n}, \beta \leq \alpha} (-1)^{\mbox{e}(\alpha)}\binom{\alpha}{\beta}X^{\beta},$$
where $\alpha=(\alpha_1,\ldots,\alpha_{2n})$.
\end{proof}

\subsection*{Acknowledgements} The   author would like to express his sincere  gratitude to Prof. Konrad Schm\"udgen for his series of lectures on moment problems  given at the Vietnam Institute for Advanced Study in Mathematics (VIASM) which motivates the author to study the  matrix-valued moment problems.    This paper was finished during the visit of the    author    at  VIASM. He thanks VIASM for financial support and hospitality.


\end{document}